\documentclass[12pt,a4paper,twoside]{amsart}
\usepackage{geometry}
\usepackage{txfonts,pxfonts}
\relax {}
\usepackage[T1]{fontenc}

\usepackage{amsmath}
\usepackage{amsthm}

\usepackage{mathtools}
\relax {}
 
\relax {}
\usepackage{subcaption}

\relax {}
\usepackage{verbatim}
\relax {}
\usepackage{xcolor}

\usepackage{hyperref}
\hypersetup{colorlinks,linkcolor={blue!75!black},citecolor={blue!75!black},urlcolor={blue!75!black}}

\relax {}

\DeclareMathOperator*{\supp}{supp}

\newcommand{\norm}[1]{\|#1\|}

\newcommand{\R}{\mathbb{R}}

\allowdisplaybreaks[4]

\mathtoolsset{showonlyrefs}

\theoremstyle{plain}

\newtheorem{theorem}{Theorem}[section] %

\newtheorem{lemma}[theorem]{Lemma}

\theoremstyle{definition}

\theoremstyle{remark}
\newtheorem{remark}[theorem]{Remark}
\newtheorem{example}[theorem]{Example}

\numberwithin{equation}{section}

\newcommand{\abs}[1]{\left\vert#1\right\vert}

\newcommand{\dint}{\,\mathrm{d}}

\title[An effective van der Corput-type method in higher dimensions]{An effective van der Corput-type method in higher dimensions}
\author{Shaozhen Xu}

\address{School of Information Engineering, Nanjing Xiaozhuang University, Nanjing 211171, China}
\email{shaozhen@nju.edu.cn}
\begin{document}
	
	\begin{abstract}
			Using a proper birational map from a smooth toric variety (adapted to the phase function of the oscillatory integral) to $\R^n$, we can effectively carry out the van der Corput-type analysis in higher dimensions. This allows us to give a direct derivation of Varchenko's decay upper bound \cite{Var76}, including its logarithmic factor. We expect that this observation may have further applications to other problems involving oscillatory integrals. 
	\end{abstract}
	
	\maketitle
	
	\setcounter{tocdepth}{1}
	
	\tableofcontents
	\section{Introduction}
	 For the oscillatory integral
	 \begin{equation}
	 	I(\lambda)=\int_{\R^n}e^{i\lambda f(x)}\varphi(x)\dint x.
	 \end{equation}
	 The classical van der Corput lemma states that
	 \begin{lemma}
	 	If $n=1$, $k\geq1$, $\supp\varphi\subset[a,b]$, $f^{(k)}(x)\geq 1$ on $[a,b]$, and in the case $k=1$, $f^{'}(x)$ is also monotone,
	 	then there is a constant $C_{\varphi}$ such that
	 	\begin{equation}
	 		|I(\lambda)|\leq C_{\varphi}\lambda^{-1/k}
	 	\end{equation}
	 \end{lemma}
	 The importance of van der Corput lemma lies in the fact that it not only provides the optimal decay rate but also yields uniform estimates. However one can not generally expect both optimal and uniform decay estimates for oscillatory integrals in higher dimensions. There is a higher-dimensional analogue of van der Corput lemma 
	 \begin{lemma}
	 	If $\frac{\partial^\alpha f}{\partial x^\alpha}\geq 1$ holds on the support of $\varphi$ for some nonzero multi index $\alpha=(\alpha_1,\cdots,\alpha_n)$, then
	 	\begin{equation}
	 		\abs{I(\lambda)}\leq C_{\varphi, f}\lambda^{-1/|\alpha|},
	 	\end{equation}
	 	where the constant also depends on higher derivatives of $f$.
	 \end{lemma}
	 It should be noted that the decay estimate above is, in general, neither optimal nor uniform. Optimal and uniform estimates are each important to analysts, albeit from different perspectives. For uniform estimates we refer the readers to \cite{CCW99} and \cite{BGZZ21}.
	 
	 If one focuses only on the optimal decay of an oscillatory integral, the problem is inherently local. In the local setting, more refined information, such as asymptotic estimates or even full asymptotic expansions, may be expected. The stationary phase principle shows that the behavior of an oscillatory integral is governed by the critical points of its phase: well-behaved critical points yield better decay, while poorly behaved ones lead to weaker decay. In the nondegenerate case, the sharp decay follows from the classical stationary phase method. For phase functions of algebraic type, such as polynomials or real-analytic functions, critical points correspond to singularities of hypersurfaces, motivating the use of algebraic-geometric techniques, in particular resolution of singularities, in analysis. 
	 
	 In the 1960s, Arnold proposed the following hypothesis: 
	 
	 \emph{For oscillatory integrals with real-analytic phase functions, the principal term of the asymptotic expansion is determined by the Newton distance of the phases.}
	 
	 Later, Varchenko confirmed the hypothesis for generic phases in the seminal work \cite{Var76}, using toric resolution of singularities. 
	 
	 For the reader's convenience, we briefly recall Varchenko's approach. One first rewrites $I(\lambda)$ in the form
	 \begin{align}
	 	I(\lambda)&=\int_{0}^{+\infty}e^{i\lambda t}J(t)\dint t+\int_{-\infty}^{0}e^{i\lambda t}J(t)\dint t\\
	 	&=\int_{0}^{+\infty}e^{i\lambda t}J(t)\dint t+\int_{0}^{+\infty}e^{-i\lambda t}J(-t)\dint t
	 \end{align}
	 where $J(t)$ is the Gelfand-Leray function
	 \begin{equation}
	 	J(t)=\int_{f=t}\varphi\cdot\omega, \quad t\neq0\text{ a regular value},
	 \end{equation}
     and $\omega$ is the Gelfand-Leray form determined by
     \begin{equation}
     	\omega\wedge \dint t=\dint x_1\cdots\dint x_n.
     \end{equation}
	Once the asymptotic expansion of $J(t)$ is known, the corresponding expansion for $I(\lambda)$ follows immediately. The function $J(t)$ is connected to the local zeta function 
	\begin{align}
		Z_f(s)&=\int_{\R^n}\abs{f(x)}^s\varphi(x)\dint x\\
		&=\int_{f\geq 0}(f(x))^s\varphi(x)\dint x+\int_{f<0}(-f(x))^s\varphi(x)\dint x\\
		&:=Z_f^+(s)+Z_f^-(s)
	\end{align}
	through inverse Mellin's transform, on a vertical line in the domain of convergence, as follows
	\begin{equation}
		J(t)=\frac{1}{2\pi i}\int_{c-i\infty}^{c+i\infty}t^{-s-1}Z_f^+(s)\dint s,\quad J(-t)=\frac{1}{2\pi i}\int_{c-i\infty}^{c+i\infty}t^{-s-1}Z_f^-(s)\dint s,\quad t>0.
	\end{equation}
	This reduces the study of $J(t)$ to the meromorphic continuation of $Z_f(s)$ to all of $\mathbb{C}$. This continuation was conjectured by Gelfand in 1954, and later proved independently by Bernstein-Gelfand \cite{BG1969} and by Atiyah \cite{Atiyah1970}. Bernstein \cite{Bernstein} subsequently gave another proof using what is now known as Bernstein-Sato polynomial, introduced in the context of $\mathcal{D}-$module theory. The proofs of both Bernstein-Gelfand and Atiyah rely on Hironaka's celebrated resolution theorem \cite{Hir64}. However, Hironaka's result is purely existential, whereas Arnold's hypothesis points directly to an explicitly computable quantity-the Newton distance of the phase function. This is precisely what Varchenko accomplished: for generic phase functions, he employed a special type of resolution to relate the principal term of the asymptotic expansion to the geometry of the Newton polyhedron. More explicitly, Varchenko obtained the asymptotic expansion
	\begin{equation}
		I(\lambda)\sim \sum_{k}\sum_{l=1}^{l_k}a_{k,l}\lambda^{-p_k}\ln^{l-1}\lambda,
	\end{equation}
	where the possible exponents lie in a finite union of rational arithmetic progressions. Under Newton nondegeneracy, the Newton distance and the principal face determine the power and logarithmic factor in Varchenko's upper bound.
	
	This leads to a somewhat awkward situation: in most applications, analysts are primarily interested in decay estimates, while Varchenko's proof also provides stronger asymptotic information. If one wishes to avoid algebraic geometry, it becomes necessary to develop analytic tools that replicate the effect of toric resolution. In this direction, Greenblatt made substantial contributions: by introducing a remarkable resolution algorithm in \cite{Gre07}, he obtained asymptotic expansions for a broad class of smooth phase functions; see also \cite{Gre09, Gre10, Gre10Res}. In a related development, Gilula \cite{Gil18} established analogous results using a purely real-variable approach. Both proofs rely on a delicate decomposition of the domain, applying van der Corput lemma locally and summing the contributions of each piece. Pursuing a purely analytic proof, however, comes at the cost of losing much of the geometric intuition that algebraic methods naturally provide. Further advances on the asymptotic behavior of oscillatory integrals and on the meromorphic continuation of local zeta functions for general smooth functions were achieved in \cite{CKN13,KN15,KN16,KN16New,KN20}. 
	
	In this work, we aim to strike an optimal balance: \emph{reducing analytic complexity while retaining the key geometric intuition}.
	
	Before presenting our main novelty and key observation, we revisit the aforementioned higher-dimensional analogue of van der Corput lemma. We briefly outline its proof, discuss its strengths and limitations, and then introduce the key improvements underlying our approach. Consider
	\begin{equation}
		I(\lambda)=\int_{\R^n}e^{i\lambda f(x)}\varphi(x)\dint x,
	\end{equation}
	where $\varphi$ is a smooth cut-off,  and assume that
	\begin{equation}
		\frac{\partial^\alpha f}{\partial x^\alpha}\geq 1.
	\end{equation}
	Following Stein\cite{Ste93}, there exists a finite collection of unit vectors $\{\xi_j\}$ such that 
	\begin{equation}
		\partial^\alpha_xf(x)=\sum_{j}a_j\left(\xi_j\cdot\nabla\right)^{\abs{\alpha}}f(x).
	\end{equation}
	Hence, at each point of the support, the lower bound on $\partial^\alpha_xf(x)$ implies that for some $\xi_j$ we have $\abs{\left(\xi_j\cdot\nabla\right)^{\abs{\alpha}}f(x)}>c$, where $c>0$ depends only on $\alpha$ and the dimension. By continuity and a smooth partition of unity on the compact support, we may fix such a direction on each sufficiently small piece. Rotating that direction to a coordinate axis, one may apply the one-dimensional van der Corput lemma and obtain the desired decay. When $\abs{\alpha}=1$, a direct integration-by-parts argument gives the same conclusion, with a constant depending on higher derivatives.
	
	This argument has two clear advantages. First, the rotation is an invertible linear transformation, so the Jacobian plays no role. Second, the directional derivatives are collected into a single direction, allowing direct use of the one-dimensional van der Corput lemma. Its limitation, however, stems from the same linearity: although we obtain a direction along which to apply van der Corput lemma, we have no idea which this direction is, and in general it does not yield the optimal decay. This reflects the broader fact that linear changes of variables do not reveal sufficient information about the singularities of the phase. By contrast, nonlinear transformations, such as blow-ups arising in algebraic geometry, capture precisely the singular structure that must be resolved. So our key observation is
	
	\emph{Applying a toric resolution of singularities to the phase function reveals explicit geometric directions adapted to its singularities, and these directions are exactly those along which one-dimensional van der Corput estimates produce the optimal decay rate of the oscillatory integral. }

\section{A Motivating Example and Generalizations}
	In this section we begin with a toy model that illustrates the central idea, and then extend the argument to the general case. Throughout, $\mathbb Z_+=\mathbb N=\{0,1,\cdots\}$, all amplitudes are smooth and compactly supported, and all decay estimates are stated for $\lambda\geq2$.
	 \begin{example}
	 	Consider the toy model $f(x_1,x_2)=x_1^2x_2$, applying van der Corput lemma along the direction of some $\xi_j$ above yields
	 	\begin{equation}
	 		|I(\lambda)|\leq C_\varphi\lambda^{-\frac{1}{3}}.
	 	\end{equation}
	 	However, the phase is more degenerate in the $x_1-$direction. If we first integrate in $x_1$ and apply the van der Corput lemma to $x_1\to x_1^2x_2$, we obtain
	 	\begin{equation}
	 		\abs{\int_{\R}e^{i\lambda x_1^2x_2}\varphi(x_1,x_2)\dint x_1}\leq C_\varphi\abs{\lambda x_2}^{-\frac{1}{2}},
	 	\end{equation}
	 	and therefore
	 	\begin{equation}
	 		\abs{I(\lambda)}\leq C_\varphi \lambda^{-\frac{1}{2}}.
	 	\end{equation}
	 \end{example}
	Thus, choosing the direction that reflects the actual singular structure of the phase yields a strictly better decay than that obtained from the linear method.
	Now consider the monomial 
	\begin{equation}
		f(x)=x_1^{\alpha_1}\cdots x_n^{\alpha_n}, \quad \alpha\in\mathbb{Z}_{+}^n
	\end{equation}
	For such a phase, it is clear that the most singular behavior occurs along those coordinate directions $x_j$ for which $\alpha_j$ is largest; these are precisely the directions in which the phase degenerates most severely. We now state a lemma that provides the decay rate for oscillatory integrals with this monomial phase.
	\begin{lemma}
		Consider
		\begin{equation}
			I(\lambda)=\int_{\R^n}e^{i\lambda x_1^{\alpha_1}\cdots x_n^{\alpha_n}}\varphi(x_1,\cdots,x_n)\dint x_1\cdots\dint x_n,
		\end{equation}
		where $\alpha\in\mathbb{Z}_{+}^n\setminus\{0\}$, $\varphi\in C_c^\infty(\R^n)$, $\supp\varphi\subset [-1,1]^n$ and denote by 
		\begin{equation}
			d=\max_j\alpha_j,\quad  M=\#\{j:\alpha_j=d\},
		\end{equation}
		then it follows
		\begin{equation}
			\abs{I(\lambda)}\leq C_\varphi\lambda^{-\frac{1}{d}}\ln^{M-1}\lambda.
		\end{equation}
	\end{lemma}
	\begin{proof}
		Variables for which $\alpha_j=0$ may first be integrated out, so we assume $\alpha_j>0$. On each orthant, replacing $x_j$ by $\pm x_j$ changes the phase only by a constant sign, and the amplitude remains smooth up to the boundary. It therefore suffices to work on the first orthant $\R_{+}^n$ and set
		\begin{equation}
			I^{+}(\lambda)=\int_{\R_{+}^n}e^{i\lambda x_1^{\alpha_1}\cdots x_n^{\alpha_n}}\varphi(x_1,\cdots,x_n)\dint x_1\cdots\dint x_n.
		\end{equation}
		To obtain an estimate for $I(\lambda)$, it suffices to prove a similar estimate for $I^{+}(\lambda)$. To analyze $I^{+}(\lambda)$, we decompose the domain depending on whether the oscillation is effective or negligible.
		This yields a splitting
		\begin{align}
			I^{+}(\lambda)&=\int_{\left\{x\in\R_{+}^n:x_1^{\alpha_1}\cdots x_n^{\alpha_n}<1/\lambda\right\}}e^{i\lambda x_1^{\alpha_1}\cdots x_n^{\alpha_n}}\varphi(x_1,\cdots,x_n)\dint x_1\cdots\dint x_n\\
			&+\int_{\left\{x\in\R_{+}^n:x_1^{\alpha_1}\cdots x_n^{\alpha_n}\geq1/\lambda\right\}}e^{i\lambda x_1^{\alpha_1}\cdots x_n^{\alpha_n}}\varphi(x_1,\cdots,x_n)\dint x_1\cdots\dint x_n\\
			&:=I_1^{+}(\lambda)+I_2^{+}(\lambda),
		\end{align}
		where $I_1^{+}(\lambda)$ is the contribution from a sublevel region (handled by size estimates), and $I_2^{+}(\lambda)$ is the contribution from the oscillatory region (handled by van der Corput-type estimates).
		Therefore
		\begin{align}
			\abs{I_1^{+}(\lambda)}&\leq \norm{\varphi}_{L^\infty}\abs{\left\{x\in[0,1]^n:x_1^{\alpha_1}\cdots x_n^{\alpha_n}<1/\lambda\right\}} \\
			&\leq C_\varphi\lambda^{-\frac{1}{d}}\ln^{M-1}\lambda.
		\end{align}
			The last inequality follows from Lemma 3.1 of \cite{Gre10}. We now proceed to treat the term $I_2^+$. Without loss of generality, we may assume $\alpha_1=d$. Observe that
		\begin{equation}
			\left\{x\in[0,1]^n:x_1^{\alpha_1}x_2^{\alpha_2}\cdots x_n^{\alpha_n}\geq 1/\lambda\right\}\subset \left\{x\in[0,1]^n:x_j\geq \lambda^{-\frac{1}{\alpha_j}},j=1,\cdots,n\right\}
		\end{equation}
		and that
		\begin{equation}
			\partial_{x_1}^d(x_1^{\alpha_1}x_2^{\alpha_2}\cdots x_n^{\alpha_n})=d!x_2^{\alpha_2}\cdots x_n^{\alpha_n}.
		\end{equation}
		Applying van der Corput lemma in the $x_1-$variable yields
		\begin{align}
			\abs{I_2^+}&\leq C_\varphi\int_{\lambda^{-\frac{1}{\alpha_n}}}^1\cdots\int_{\lambda^{-\frac{1}{\alpha_2}}}^1\left(d!\lambda x_2^{\alpha_2}\cdots x_n^{\alpha_n}\right)^{-\frac{1}{d}}\dint x_2\cdots\dint x_n\\
			&\leq C_\varphi\lambda^{-\frac{1}{d}}\ln^{M-1}\lambda.
		\end{align}
	\end{proof}

	Now we state a more general lemma.
	\begin{lemma}\label{KeyLemma}
		Consider the oscillatory integral
		\begin{equation}
			I(\lambda)=\int_{\R^n}e^{i\lambda x_1^{\alpha_1}\cdots x_n^{\alpha_n}}\abs{x_1^{\beta_1}\cdots x_n^{\beta_n}}\varphi(x_1,\cdots,x_n)\dint x_1\cdots\dint x_n,
		\end{equation}
		where $\alpha, \beta\in\mathbb{Z}_{+}^n$, $\alpha\neq0$, $\varphi\in C_c^\infty(\R^n)$, and $\supp\varphi\subset [-1,1]^n$. Set 
		\begin{equation}
			d=\max_j\frac{\alpha_j}{\beta_j+1},\quad M=\#\left\{j:\frac{\alpha_j}{\beta_j+1}=d\right\}.
		\end{equation}
		Then one has the decay estimate
		\begin{equation}
			\abs{I(\lambda)}\leq C_\varphi\lambda^{-\frac{1}{d}}\ln^{M-1}\lambda.
		\end{equation}
	\end{lemma}
	\begin{proof}
		We first treat the case $d\geq 1$. The case $d<1$ can be reduced to this one by an integration-by-parts argument. Variables with $\alpha_j=0$ may first be integrated out, so we assume $\alpha_j>0$. Similar to the preceding lemma, the estimate of $I(\lambda)$ reduces to estimating $I^{+}(\lambda)$ and $I^{-}(\lambda)$. We write
		\begin{align}
			I_1^{+}(\lambda)&=\int_{\left\{x\in\R_{+}^n:x_1^{\alpha_1}\cdots x_n^{\alpha_n}<1/\lambda\right\}}e^{i\lambda x_1^{\alpha_1}\cdots x_n^{\alpha_n}}x_1^{\beta_1}\cdots x_n^{\beta_n}\varphi(x_1,\cdots,x_n)\dint x_1\cdots\dint x_n,\\
			I_2^{+}(\lambda)&=\int_{\left\{x\in\R_{+}^n:x_1^{\alpha_1}\cdots x_n^{\alpha_n}\geq1/\lambda\right\}}e^{i\lambda x_1^{\alpha_1}\cdots x_n^{\alpha_n}}x_1^{\beta_1}\cdots x_n^{\beta_n}\varphi(x_1,\cdots,x_n)\dint x_1\cdots\dint x_n.\\
		\end{align}
		To estimate $I_1^{+}(\lambda)$, we perform the change of variables
		\begin{equation}
			x_j=y_j^{\frac{1}{\beta_j+1}}, \quad j=1,\cdots,n.
		\end{equation}
		Then
		\begin{equation}
			\abs{I_1^{+}(\lambda)}\leq C_\varphi \int_{\left\{y\in[0,1]^n:y_1^{\frac{\alpha_1}{\beta_1+1}}\cdots y_n^{\frac{\alpha_n}{\beta_n+1}}<1/\lambda\right\}}\dint y_1\cdots\dint y_n,
		\end{equation}
		and by Lemma 3.1 in \cite{Gre10}, we obtain
		\begin{equation}
			\abs{I_1^{+}(\lambda)}\leq C_\varphi\lambda^{-\frac{1}{d}}\ln^{M-1}\lambda.
		\end{equation}
		We now deal with $I_2^+(\lambda)$. Without loss of generality we assume that $\frac{\alpha_1}{\beta_1+1}=d$. Observe that 
		\begin{align}
			I_2^{+}(\lambda)&=\int_{\left\{x\in\R_{+}^n:x_1^{\alpha_1}\cdots x_n^{\alpha_n}\geq1/\lambda\right\}}e^{i\lambda x_1^{\alpha_1}\cdots x_n^{\alpha_n}}x_1^{\beta_1}\cdots x_n^{\beta_n}\varphi(x_1,\cdots,x_n)\dint x_1\cdots\dint x_n\\
			&=\int_{D(x_2,\cdots,x_n)}\left[\int_{L(x_2,\cdots,x_n)}^{1}e^{i\lambda x_1^{\alpha_1}\cdots x_n^{\alpha_n}}x_1^{\beta_1}\cdots x_n^{\beta_n}\varphi(x_1,\cdots,x_n)\dint x_1\right]\dint x_2\cdots\dint x_n\\
			\intertext{Taking absolute values, we have}
            \abs{I_2^+(\lambda)}&\leq \int_{\lambda^{-\frac{1}{\alpha_n}}}^1\cdots\int_{\lambda^{-\frac{1}{\alpha_2}}}^1\abs{\int_{L(x_2,\cdots,x_n)}^{1}e^{i\lambda x_1^{\alpha_1}\cdots x_n^{\alpha_n}}x_1^{\beta_1}\cdots x_n^{\beta_n}\varphi(x_1,\cdots,x_n)\dint x_1}\dint x_2\cdots\dint x_n,
		\end{align}
		where $L(x_2,\cdots,x_n)=\left(\lambda x_2^{\alpha_2}\cdots x_n^{\alpha_n}\right)^{-\frac{1}{\alpha_1}}$ and the inner integral is understood to be zero when $L>1$. The outer domain $D(x_2,\cdots,x_n)$ is contained in the box $[\lambda^{-\frac{1}{\alpha_2}}, 1]\times \cdots\times [\lambda^{-\frac{1}{\alpha_n}}, 1]$.

		By integration by parts against the smooth amplitude, it suffices to estimate the following primitive uniformly with its upper endpoint replaced by any $T\in[L(x_2,\cdots,x_n),1]$:
		\begin{equation}
			\int_{L(x_2,\cdots,x_n)}^{1}e^{i\lambda x_1^{\alpha_1}\cdots x_n^{\alpha_n}}x_1^{\beta_1}\dint x_1.
		\end{equation}
		 A direct computation yields
		\begin{align}
			&\int_{L(x_2,\cdots,x_n)}^{1}e^{i\lambda x_1^{\alpha_1}\cdots x_n^{\alpha_n}}x_1^{\beta_1}\dint x_1\\
			& =\left.\frac{e^{i\lambda x_1^{\alpha_1}\cdots x_n^{\alpha_n}}x_1^{\beta_1}}{i\alpha_1\lambda x_1^{\alpha_1-1}\cdots x_n^{\alpha_n}}\right|_{L(x_2 ,\cdots,x_n)}^1- \frac{\beta_1+1-\alpha_1}{i\alpha_1\lambda x_2^{\alpha_2}\cdots x_n^{\alpha_n}}\int_{L(x_2,\cdots,x_n)}^{1}e^{i\lambda x_1^{\alpha_1}\cdots x_n^{\alpha_n}}x_1^{\beta_1-\alpha_1}\dint x_1.
		\end{align}
		Taking absolute values in the last integral and integrating $x_1^{\beta_1-\alpha_1}$ (we note that when $d=1$, the coefficient $\beta_1+1-\alpha_1=0$, and the last integral does not appear), we obtain
		\begin{equation}
			\abs{\int_{L(x_2,\cdots,x_n)}^{1}e^{i\lambda x_1^{\alpha_1}\cdots x_n^{\alpha_n}}x_1^{\beta_1}\dint x_1}\leq C\lambda^{-1}L(x_2,\cdots,x_n)^{\beta_1+1-\alpha_1}\left(x_2^{\alpha_2}\cdots x_n^{\alpha_n}\right)^{-1}.
		\end{equation} 
		The same estimate holds with upper endpoint $T$. Integrating against $\partial_{x_1}\varphi$ therefore restores the smooth amplitude. Hence
        \begin{align}
            &\abs{\int_{L(x_2,\cdots,x_n)}^{1}e^{i\lambda x_1^{\alpha_1}\cdots x_n^{\alpha_n}}x_1^{\beta_1}x_2^{\beta_2}\cdots x_n^{\beta_n}\varphi(x_1,\cdots,x_n)\dint x_1}\\
			&\leq C_\varphi\lambda^{-1}L(x_2,\cdots,x_n)^{\beta_1+1-\alpha_1}x_2^{\beta_2-\alpha_2}\cdots x_n^{\beta_n-\alpha_n}\\
			&=C_\varphi\lambda^{-1}\left(\lambda x_2^{\alpha_2}\cdots x_n^{\alpha_n}\right)^{-\frac{\beta_1+1-\alpha_1}{\alpha_1}}x_2^{\beta_2-\alpha_2}\cdots x_n^{\beta_n-\alpha_n}\\
			&=C_\varphi\lambda^{-1/d}x_2^{\beta_2-\alpha_2/d}\cdots x_n^{\beta_n-\alpha_n/d}.
		\end{align}
		Consequently,
		\begin{align}
			\abs{I_2^+(\lambda)}&\leq C_\varphi\lambda^{-1/d} \int_{\lambda^{-\frac{1}{\alpha_n}}}^1\cdots\int_{\lambda^{-\frac{1}{\alpha_2}}}^1x_2^{\beta_2-\alpha_2/d}\cdots x_n^{\beta_n-\alpha_n/d}\dint x_2\cdots\dint x_n\\
			&\leq C_\varphi\lambda^{-\frac{1}{d}}\ln^{M-1}\lambda.
		\end{align}
		This completes the proof for the case $d\geq 1$. 
		
		We now consider the case $d<1$. The key point is that the higher order of vanishing of the weight $x^\beta$ near the origin allows us to perform repeated integration by parts without incurring any loss. Without loss of generality we assume that
		\begin{equation}
			\frac{\alpha_1}{\beta_1+1}=d.
		\end{equation}
		Choosing an integer $K\geq 0$ such that
		\begin{equation}
			\beta_1+1-K\alpha_1> 0 \quad 
			\text{and}\quad \beta_1+1-(K+1)\alpha_1\leq 0.
		\end{equation}
		\begin{align}
		I^{+}(\lambda)&=\int_{\R_{+}^n}e^{i\lambda x_1^{\alpha_1}\cdots x_n^{\alpha_n}}x_1^{\beta_1}\cdots x_n^{\beta_n}\varphi(x_1,\cdots,x_n)\dint x_1\cdots\dint x_n\\
		&=\int_{\R_{+}^{n-1}}\left[\int_{\R_{+}}e^{i\lambda x_1^{\alpha_1}\cdots x_n^{\alpha_n}}\varphi(x_1,\cdots,x_n)x_1^{\beta_1}\dint x_1\right]x_2^{\beta_2}\cdots x_n^{\beta_n}\dint x_2\cdots\dint x_n.
	\end{align}
	Applying $K$ integrations by parts, all boundary terms vanish. To keep track of derivatives of the amplitude, put
	\begin{equation}
		\varphi_K(x)=\prod_{r=1}^K
		\left(\beta_1+1-r\alpha_1+x_1\partial_{x_1}\right)\varphi(x).
	\end{equation}
	For fixed $x'=(x_2,\cdots,x_n)\in(0,\infty)^{n-1}$, let $\tau=\lambda x_2^{\alpha_2}\cdots x_n^{\alpha_n}$. More precisely, we obtain
	\begin{align}
		&\int_0^\infty e^{i\tau x_1^{\alpha_1}}x_1^{\beta_1}\varphi(x)\dint x_1\\
		&\quad=\left(-\frac{1}{i\tau\alpha_1}\right)^K
		\int_0^\infty e^{i\tau x_1^{\alpha_1}}
		x_1^{\beta_1-K\alpha_1}\varphi_K(x)\dint x_1.
	\end{align}
	Indeed, the boundary term in the $r$th step contains $x_1^{\beta_1+1-r\alpha_1}$, whose exponent is strictly positive for $1\leq r\leq K$. The upper boundary terms vanish by compact support. The function $\varphi_K$ is still smooth and compactly supported, and its required seminorms are controlled by finitely many derivatives of $\varphi$.

	We also note that when $1/d$ is an integer, the condition $\beta_1+1-K\alpha_1>0$ means $K=1/d-1$, so no further integration by parts is necessary and we directly reduce to the threshold case where the new ratio equals $1$. Consequently,
	\begin{equation}
		I^+(\lambda)=\left(-\frac{1}{i\lambda\alpha_1}\right)^K
		\int_{\R_+^n}e^{i\lambda x_1^{\alpha_1}\cdots x_n^{\alpha_n}}
		\prod_{j=1}^n x_j^{\beta_j-K\alpha_j}\varphi_K(x)\dint x.
	\end{equation}
	By construction,
	\begin{equation}
		\beta_j+1-K\alpha_j>0,\qquad j=1,\cdots,n,
	\end{equation}
	since $(\beta_j+1)/\alpha_j\geq1/d>K$. Thus all the new integer weight exponents $\beta_j-K\alpha_j$ are nonnegative. By construction, we then have
	\begin{equation}
		\frac{\alpha_1}{\beta_1+1-K\alpha_1}=\max_j \frac{\alpha_j}{\beta_j+1-K\alpha_j}\geq 1.
	\end{equation}
	Thus, applying the same argument as in the case $d\geq 1$, we obtain the desired upper bound.
	\end{proof}
	
	\section{Notations and the Main Theorem}
	In this section we introduce several notions about real-analytic functions that will be used throughout the paper and state the main result to be proved.
	
	A function $f:\R^n\to \R$ is said to be real-analytic near the origin if, in some neighborhood of 0, it admits a convergent Taylor expansion 
	\begin{equation}
		f(x)=\sum_{\alpha\in \mathbb{N}^n} a_\alpha x^\alpha, \quad \alpha=(\alpha_1,\cdots,\alpha_n).
	\end{equation}
	
	The \emph{Newton polyhedron} of $f$ is defined as
	\begin{equation}
		\mathcal{N}_f:=\text{convex hull of}\left\{\alpha+\R_+^n: a_\alpha\neq 0\right\}.
	\end{equation}
		The associated \emph{Newton distance} is defined by
	\begin{equation}
		d_f:=\inf\left\{t: (t,\cdots,t)\in \mathcal{N}_f\right\}.
	\end{equation}	
	
	The face of $\mathcal{N}_f$ that contains the point $(d_f,\cdots,d_f)$ and has minimal possible dimension is called the \emph{principal face} of $f$.
	
	For any compact face $\gamma\subset\mathcal{N}_f$, the corresponding \emph{$\gamma-$part} of $f$ is 
	\begin{equation}
		f_\gamma=\sum_{\alpha\in \gamma}a_\alpha x^\alpha.
	\end{equation}
	
	A compact face $\gamma$ is said to be \emph{$\R-$nondegenerate} if 
	\begin{equation}
		\nabla f_\gamma\neq 0 \quad\text{for all } x\in (\R\backslash 0)^n,
	\end{equation}
	 that is, the gradient of $f_\gamma$ does not vanish  outside the coordinate hyperplanes.
	We say $f$ is $\R-$nondegenerate if all compact faces of its Newton polyhedron are $\R-$nondegenerate.

	Now we are ready to state the main theorem.
	\begin{theorem}\label{MainTheorem}
		Let
		\begin{equation}
			I(\lambda)=\int_{\R^n}e^{i\lambda f(x)}\varphi(x)\dint x,
		\end{equation}
		where $f$ is real-analytic, not identically zero, and $\R-$nondegenerate near the origin, satisfying $f(0)=0, \nabla f(0)=0$. Assume that $\varphi\in C_c^\infty(\R^n)$ is supported in a sufficiently small neighborhood of $0$. If the codimension of principal face is $k$, then, for $\lambda\geq2$,
		\begin{equation}
			\abs{I(\lambda)}\leq C_\varphi \lambda^{-\frac{1}{d_f}}\ln^{k-1}\lambda.
		\end{equation}
	\end{theorem}
	\begin{remark}
		Even when $1/d_f$ is an integer (in particular, when $d_f=1$), the same upper bound with logarithmic exponent $k-1$ holds, in agreement with Gilula's upper bound \cite{Gil18}. It should be noted that this is an upper-bound statement; cancellation may reduce the logarithmic order of the actual leading term.
	\end{remark}
	If we carry out the argument developed in this paper, we can also recover the following real-analytic case of the weighted estimate of Gilula (2018).
	\begin{theorem}[Gilula 2018 \cite{Gil18}]\label{WeightedTheorem}
		Consider the damped oscillatory integral
		\begin{equation}
			I(\lambda)=\int_{\R^n}e^{i\lambda f(x)}x^\beta\varphi(x)\dint x.
		\end{equation}
		where $f$ satisfies the assumptions of Theorem \ref{MainTheorem}, $\beta\in\mathbb N^n$, and $\varphi\in C_c^\infty(\R^n)$ is supported in a sufficiently small neighborhood of $0$. Then, for $\lambda\geq2$,
		\begin{equation}
			\abs{I(\lambda)}\leq C_\varphi \lambda^{-\lfloor\beta+\bf{1}\rfloor}\ln^{d_\beta-1}\lambda,
		\end{equation}
		where $\lfloor\beta+\bf{1}\rfloor$ is the maximal number $c$ such that $\frac{\beta+\bf{1}}{c}$ is contained in $\mathcal{N}_f$ and $d_\beta$ is the greatest codimension over all faces of $\mathcal{N}_f$ containing $\frac{\beta+\bf{1}}{\lfloor\beta+\bf{1}\rfloor}$.
	\end{theorem}
	
	\section{Key Facts About Toric Varieties}
	
	Although this material is standard for algebraic geometers, we include it here for the convenience of readers,  especially analysts who may be less familiar with the algebraic-geometric background.
	
	A toric variety is an object in algebraic geometry that plays a central role in making explicit the correspondence between combinatorial geometry and algebraic geometry. It provides an ideal testing ground for general theories, since many geometric and analytic phenomena can be translated into concrete combinatorial data.
	
	Our ultimate goal is to construct a toric variety associated with the Newton polyhedron of the phase function. Before doing so, we introduce the basic objects underlying this construction: cones and fans.
	
	In this paper, we only consider \emph{rational cones} $\sigma\subset \R^n$, which are generated by finitely many elements of $\mathbb{Z}^n$. That is, given a collection of vectors $\{v_1,\cdots, v_m\}$ in $\mathbb{Z}^n$, the cone they generate is 
	\begin{equation}
		\sigma=\left\{\mu_1v_1+\cdots+\mu_mv_m: \mu_1,\cdots,\mu_m \geq0\right\}.
	\end{equation}
	All cones considered below are strongly convex, meaning that they contain no nonzero linear subspace. The \emph{skeletons} of a cone $\sigma$ is the set of primitive vectors that generate $\sigma$; that is, the minimal nonzero integer vectors along the edges of $\sigma$. 
	
	A rational cone is called \emph{simplicial} if the vectors making up its skeleton are linearly independent.
	
	The \emph{fan}  $\Sigma$ in $\R^n$ is a finite collection of rational polyhedral cones satisfying the following conditions:
	\begin{itemize}
		\item Every face of a cone in $\Sigma$ is also a cone in $\Sigma$;
		\item The intersection of any two cones in $\Sigma$ is also a face of each.
	\end{itemize}
	
	For a fan $\Sigma$, the union $\abs{\Sigma}:=\bigcup_{\sigma\in\Sigma}\sigma$ is called the \emph{support} of $\Sigma$.
	
	A fan $\Sigma$ is said to be \emph{simple} if each cone in $\Sigma$ is simplicial and the skeletons of any cone can be extended to a basis of the integer lattice of the whole space.
	
	An important fact about fan is the existence of \emph{simple subdivisions}. Given a fan $\Sigma_0$ , there exists a fan $\Sigma$, called a simple subdivision of $\Sigma_0$, satisfying
	\begin{itemize}
		\item $\Sigma_0$ and $\Sigma$ have the same support;
		\item Each cone of $\Sigma$ is contained in some cone of $\Sigma_0$;
		\item $\Sigma$ is simple.
	\end{itemize}
	
	Such subdivisions are fundamental in toric geometry, particularly for resolving singularities while maintaining control over the combinatorial structure.
	
	We now define a rational map that is central to the construction of a toric variety. Let $\Sigma$ be a simple fan whose support is the closed positive orthant $\R_+^n=[0,\infty)^n$. Consider an $n-$dimensional cone $\sigma\in\Sigma$ generated by skeletons consisting of vectors
	\begin{equation}
		\alpha^1=(\alpha_1^1,\cdots,\alpha_n^1),\cdots,
		\alpha^n=(\alpha_1^n,\cdots,\alpha_n^n).
	\end{equation}
	Given such a cone $\sigma$, we associate a copy of $\R^n$ denoted by $\R^n(\sigma)$. Define the map
	\begin{equation}
		\pi(\sigma):\R^n(\sigma)\to \R^n:\quad \pi(\sigma)(y_1,\cdots,y_n)=(x_1,\cdots,x_n),
	\end{equation}
	where
	\begin{equation}
		x_j=y_1^{\alpha_j^1}\cdots y_n^{\alpha_j^n}, \quad j=1,\cdots,n.
	\end{equation}
	The toric variety $X_\Sigma$ is obtained by gluing the copies $\R^n(\sigma)$ for the $n$-dimensional cones of $\Sigma$. More explicitly, let $\tau=\sigma\cap\sigma'$. In the chart $\R^n(\sigma)$, the open set corresponding to $\tau$ consists of the points satisfying $y_j\neq0$ whenever the ray generated by $\alpha^j$ is not contained in $\tau$.
	Let $A_\sigma$ have columns $\alpha^1,\cdots,\alpha^n$. On these common-face open sets the transition map is
	\begin{equation}
		y'_k=\prod_{j=1}^n y_j^{B_{kj}},
		\qquad B=A_{\sigma'}^{-1}A_\sigma\in GL(n,\mathbb Z).
	\end{equation}
	Negative exponents occur only in coordinates that are nonzero on the overlap. The transition map and its inverse are therefore smooth, including at the allowed zero coordinates. On the real torus $(\R\backslash\{0\})^n$, this transition is $\pi(\sigma')^{-1}\circ\pi(\sigma)$. Gluing by these transitions gives the smooth real toric variety $X_\Sigma$.
	
	The key properties of toric varieties for our purposes are summarized as follows:
	
	\textbf{Facts:}
	\begin{itemize}
		\item \emph{If the fan $\Sigma$ is simple, then the associated toric variety $X_\Sigma$ is smooth.}
		\item \emph{If $\Sigma$ subdivides the fan consisting of the closed positive orthant and its faces, with the same support, then the resulting map $\pi:X_\Sigma\to\R^n$ is proper, that is, the preimage of a compact set is compact. It is a birational map and is an isomorphism over $(\R\backslash\{0\})^n$.}
	\end{itemize}
	Based on the preceding facts, we are now ready to construct the smooth toric variety associated with the Newton polyhedron.
	Consider the real-analytic function introduced in Section 3. For a vector $\xi\in\R_+^n$, write
	\begin{equation}
		l(\xi)=\min_{\nu\in\mathcal N_f}\langle\xi,\nu\rangle,\qquad
		F(\xi)=\{\nu\in\mathcal N_f:\langle\xi,\nu\rangle=l(\xi)\}.
	\end{equation}
	For a facet, its supporting normal can be chosen to be a primitive vector in $\mathbb Z_+^n$. More generally, for every nonempty face $\gamma$ of $\mathcal N_f$, including noncompact faces, define
	\begin{equation}
		N_\gamma=\{\xi\in\R_+^n:\gamma\subset F(\xi)\}.
	\end{equation}
	These cones form the normal fan $\Sigma_0$ of $\mathcal N_f$, with support $\R_+^n$. The cone $N_\gamma$ has dimension equal to the codimension of $\gamma$. The fan $\Sigma_0$ admits a simple subdivision $\Sigma$ with the same support. Consequently, $\Sigma$ determines a smooth real toric variety $X_\Sigma$ and a proper birational map $\pi:X_\Sigma\to\R^n$. The standard facts used here are also discussed in \cite{Nose2012}.
	
	\section{Proof of the Main Theorem}
	\begin{proof}
		We transform the oscillatory integral over $\R^n$ into an integral over the smooth toric variety $X_\Sigma$:
		\begin{equation}
			I(\lambda)=\int_{\R^n}e^{i\lambda f(x)}\varphi(x)\dint x=\int_{X_\Sigma}\pi^*\left(e^{i\lambda f}\varphi\right) \pi^*(\dint x).
		\end{equation}
	The change-of-variables identity follows from the isomorphism over $(\R\backslash\{0\})^n$; the omitted coordinate sets have measure zero. Since $\pi$ is proper and $\varphi$ has compact support, the pulled-back amplitude has compact support. It suffices to integrate over the restricted domain $X_\Sigma\bigcap \pi^{-1}(\supp \varphi)$. Using a finite collection of charts $\{U_k\}$ near this compact set and a smooth partition of unity with compactly supported chart pieces, we can decompose $I(\lambda)$ as
	\begin{equation}
		I(\lambda)=\sum_k\int_{X_\Sigma\bigcap \pi^{-1}(\supp \varphi)\bigcap U_k}\pi^*\left(e^{i\lambda f}\varphi\right) \pi^*(\dint x).
	\end{equation}
	Let $k$ be fixed, and set
	\begin{equation}
		V=X_\Sigma\bigcap \pi^{-1}(\supp \varphi)\bigcap U_k,
	\end{equation}
	and for convenience, we continue to denote by $I(\lambda)$ the integral over $V$. On such a chart $U_k$, the coordinates are given by the rational map associated with an $n-$dimensional cone $\sigma\in \Sigma$, whose skeletons are 
	\begin{equation}
		\alpha^1=(\alpha_1^1,\cdots,\alpha_n^1),\cdots,
		\alpha^n=(\alpha_1^n,\cdots,\alpha_n^n).
	\end{equation}
	Direct computation, using $\abs{\det A_\sigma}=1$, shows that the absolute Jacobian of $\pi(\sigma)$ is 
	\begin{equation}
		\abs{J_{\pi(\sigma)}(y)}=\abs{y_1}^{\abs{\alpha^1}-1}\cdots \abs{y_n}^{\abs{\alpha^n}-1},\quad \abs{\alpha^j}:=\alpha_1^j+\cdots+\alpha_n^j.
	\end{equation}
	We now examine the form of the phase function on the chart $U_k$. By construction of the skeleton vectors $\alpha^1,\cdots,\alpha^n$, associated with the cone $\sigma$,
	the corresponding supporting faces have a common vertex of $\mathcal N_f$. Consequently, the pullback of $f$ under the associated monomial map $\pi(\sigma)$ may be written as
	\begin{equation}
		f\circ\pi(\sigma)(y)=y_1^{l(\alpha^1)}\cdots y_n^{l(\alpha^n)}f_\sigma(y),
	\end{equation}
	where $f_\sigma$ is real-analytic on a neighborhood of the fiber in this chart and $f_\sigma(0)\neq0$. Indeed, all normalized Taylor exponents $\langle\alpha^j,\nu\rangle-l(\alpha^j)$ are nonnegative, and the common vertex gives the nonzero constant term.
	
	For a point $P\in\pi(\sigma)^{-1}(0)$, let $J=\{j:P_j=0\}$. Then
	\begin{equation}
		w_J:=\sum_{j\in J}\alpha^j\in(0,\infty)^n,
	\end{equation}
	since every coordinate monomial of $\pi(\sigma)$ vanishes at $P$. The common supporting face
	\begin{equation}
		\gamma_J=\{\nu\in\mathcal N_f:
		\langle\alpha^j,\nu\rangle=l(\alpha^j),\ j\in J\}
	\end{equation}
	is nonempty because $\Sigma$ refines the normal fan. It is exposed by $w_J$ and hence is compact: on this face, $0\leq\nu_i\leq l(w_J)/(w_J)_i$. In particular,
	\begin{equation}
		\sum_{j\in J}l(\alpha^j)=l(w_J)>0,
	\end{equation}
	since $f(0)=0$ and $f$ is not identically zero.
	For this actual zero-coordinate set, put
	\begin{equation}
		T_J=\{y:y_j=0,\ j\in J\}.
	\end{equation}
	Then $\pi(\sigma)(T_J)=0$, while the coordinates indexed by $J^c$ are bounded away from zero in a sufficiently small neighborhood of $P$.

	The $\R$-nondegeneracy of $f$ yields the following structural property: if $f_\sigma(P)=0$, there exists an index $j'\notin J$ such that
	\begin{equation}
		\frac{\partial f_\sigma}{\partial y_{j'}}(P)\neq0.
	\end{equation}
	This fact originates in Varchenko's analysis \cite{Var76}, and a detailed exposition may be found in \cite{Nose2012}.

	Choose small neighborhoods centered at the points of $\pi^{-1}(0)$, separating centers with $f_\sigma(P)\neq0$ from those with $f_\sigma(P)=0$. Properness implies that, after shrinking $\supp\varphi$ around the origin, finitely many of these neighborhoods cover $\pi^{-1}(\supp\varphi)$. Refining the partition of unity, we may work in such neighborhoods. On each piece, after relabelling, $J=\{1,\cdots,s\}$, where $s$ may depend on the piece.
	
	We proceed to estimate
	\begin{equation}
		I(\lambda)=\int_{V}e^{i\lambda y_1^{l(\alpha^1)}\cdots y_n^{l(\alpha^n)}f_\sigma(y)}\abs{y_1}^{\abs{\alpha^1}-1}\cdots \abs{y_n}^{\abs{\alpha^n}-1}\varphi_k(y)\varphi\circ\pi(\sigma)(y) \dint y
	\end{equation} 
	By a partition of unity, we decompose $V$ into finitely many pieces.
	\begin{equation}
		1=\sum_k\psi_k+\sum_l\eta_l.
	\end{equation}
	Each $\psi_k$ is supported in a sufficiently small neighborhood of a fiber point $P$ with $f_\sigma(P)\neq0$, where $f_\sigma$ is bounded away from zero. Each $\eta_l$ is supported near a fiber point where $f_\sigma$ vanishes, and the $\eta_l$ cover all zeroes meeting the support of the pulled-back amplitude. Inserting this decomposition into $I(\lambda)$ we obtain corresponding integrals $I_k(\lambda)$ and $I_l(\lambda)$. It suffices to estimate each piece separately.
	
	On the support of $\psi_k$, the function $f_\sigma(y)$ is bounded from zero. At least one $l(\alpha^j)$ with $j\in J$ is positive. After relabelling assume $l(\alpha^1)>0$, and introduce
	\begin{equation}
		u_j=y_j\quad(j\neq1),\qquad
		u_1=y_1\abs{f_\sigma(y)}^{1/l(\alpha^1)}.
	\end{equation}
	This is an invertible change of variables on a sufficiently small neighborhood, since $P_1=0$ and
	\begin{equation}
		\frac{\partial u_1}{\partial y_1}(P)
		=\abs{f_\sigma(P)}^{1/l(\alpha^1)}\neq0.
	\end{equation}
	The change in the weight and the Jacobian is absorbed into a smooth compactly supported amplitude. After this substitution, the phase becomes a pure monomial up to the constant sign of $f_\sigma$, and
	\begin{equation}
		I_k(\lambda)=\int_{\R^n}
		e^{\pm i\lambda u_1^{l(\alpha^1)}\cdots u_n^{l(\alpha^n)}}
		\prod_{j=1}^n\abs{u_j}^{\abs{\alpha^j}-1}\Phi_k(u)\dint u.
	\end{equation}
	Applying Lemma \ref{KeyLemma}, we use the geometric inequality
	\begin{equation}\label{CriticalDirections}
		l(\alpha^j)\leq d_f\abs{\alpha^j},\qquad j=1,\cdots,n.
	\end{equation}
	Indeed, $q_0=(d_f,\cdots,d_f)\in\mathcal N_f$. Equality means that the supporting face contains $q_0$, and therefore contains its smallest face $\gamma_0$. All such normals lie in the same normal cone $N_{\gamma_0}$, whose dimension is $k$. Since $\alpha^1,\cdots,\alpha^n$ are linearly independent, at most $k$ indices attain equality.
	Let $d_\sigma=\max_j l(\alpha^j)/\abs{\alpha^j}$. If $d_\sigma=d_f$, the multiplicity in Lemma \ref{KeyLemma} is at most $k$. If $d_\sigma<d_f$, the strictly better power of $\lambda$ absorbs the logarithmic factor. A fixed rescaling of the compact support does not change these exponents. Consequently,
	\begin{equation}
		\abs{I_k(\lambda)}\leq C_{\Phi_k}\lambda^{-1/d_f}\ln^{k-1}\lambda.
	\end{equation}
	Hence the required bound follows for each nonzero-unit piece.
	
	Since the region of integration is localized in the preceding fiber neighborhoods, we may assume that each $\eta_l$ is supported in a small neighborhood of a point
	\begin{equation}
		w^l=(w_1^l,\cdots,w_n^l)\in\pi^{-1}(0)\cap\{y:f_\sigma(y)=0\},\qquad J_l=\{j:w_j^l=0\}=\{1,\cdots,s\}.
	\end{equation} 
	By the discussion above, we may, without loss of generality, take the index $j'$ appearing earlier to be $s+1\notin J_l$, so that
	\begin{equation}
		\partial_{y_{s+1}}f_\sigma(w^l)\neq 0 \quad \text{and } \quad w_j^l\neq 0 \text{ for } j\notin J_l.
	\end{equation} 
	Thus, by shrinking the support of $\eta_l$ if necessary, we may assume that for all $y=(y_1,\cdots,y_n)\in \supp (\eta_l)$, one has $y_{s+1}\neq 0$.
	Consequently, we may introduce the following change of variables:
	\begin{equation}
	u_j=y_j \quad  (j\neq s+1),\quad 	u_{s+1}=y_{s+1}^{l(\alpha^{s+1})}f_\sigma(y).
	\end{equation}
	This transformation is smooth and invertible, since its derivative in the $y_{s+1}$ direction at $w^l$ equals $(w_{s+1}^l)^{l(\alpha^{s+1})}\partial_{y_{s+1}}f_\sigma(w^l)\neq0$. Since $y_{s+1}$ is bounded away from $0$ on $\supp(\eta_l)$, the smooth nonvanishing factor $\abs{y_{s+1}}^{\abs{\alpha^{s+1}}-1}$ together with the Jacobian of the transformation can be absorbed into the smooth cut-off function $\Psi_l(u)$. In the new coordinates, the phase again reduces to a monomial. Specifically,
	\begin{equation}
		I_l(\lambda)=\int_{\R^n}e^{i\lambda u_1^{l(\alpha^1)}\cdots u_s^{l(\alpha^s)} u_{s+1} u_{s+2}^{l(\alpha^{s+2})} \cdots u_n^{l(\alpha^n)}}\prod_{j\neq s+1}\abs{u_j}^{\abs{\alpha^j}-1}\Psi_l(u)\dint u,
	\end{equation} 
	where $\Psi_l\in C_c^\infty(\R^n)$ is smooth across $u_{s+1}=0$. The other nonvanishing coordinate factors may be left in the displayed monomial.
	Set
	\begin{equation}
		d_l=\max_{j\neq s+1}\frac{l(\alpha^j)}{\abs{\alpha^j}}\leq d_f,\qquad
		M_l=\#\left\{j\neq s+1:\frac{l(\alpha^j)}{\abs{\alpha^j}}=d_l\right\}.
	\end{equation}
	At least one of these phase exponents is positive, since $\sum_{j\in J_l}l(\alpha^j)>0$ and $s+1\notin J_l$. In the case $d_f>1$, the exponent ratio of $u_{s+1}$ is $1$. The maximum for the full monomial is therefore $\max\{1,d_l\}$. If this equals $d_f$, the transverse variable is not critical and the multiplicity is at most $k$ by the preceding normal-cone argument. If it is strictly smaller than $d_f$, the gain in the power of $\lambda$ absorbs the logarithmic factor. Lemma \ref{KeyLemma} gives the required bound.
	
	Suppose now $d_f\leq1$. To avoid a collision between the exponent ratio $1$ of $u_{s+1}$ and those of the other variables, we exploit the fact that the phase is linear in $u_{s+1}$ and the amplitude is smooth across $u_{s+1}=0$. We keep this variable on the whole real line. Only the outer variables are split into orthants; their signs change the oscillatory parameter by a constant sign. It suffices to treat the positive outer orthant.
	Denote
	\begin{equation}
		t=u_{s+1},\quad u'=(u_1,\cdots,u_s,u_{s+2},\cdots,u_n),\quad
		q(u')=\prod_{j\neq s+1}u_j^{a_j},\quad
		w(u')=\prod_{j\neq s+1}u_j^{b_j-1},
	\end{equation}
	where $a_j=l(\alpha^j)$ and $b_j=\abs{\alpha^j}$. A fixed rescaling places the outer support in $[0,1]^{n-1}$; any resulting positive constant in the phase is absorbed into $\lambda$, without changing the exponents. The bounded range of the rescaled parameter is covered by increasing the constant.
	As in the proof of Lemma \ref{KeyLemma}, split the outer domain according to whether $q(u')<1/\lambda$ or $q(u')\geq1/\lambda$:
	\begin{align}
		I_{l,1}^+(\lambda)
		&=\int_{\substack{u'\in[0,1]^{n-1}\\q(u')<1/\lambda}}
		\left[\int_\R e^{i\lambda tq(u')}\Psi_l(u',t)\dint t\right]w(u')\dint u',\\
		I_{l,2}^+(\lambda)
		&=\int_{\substack{u'\in[0,1]^{n-1}\\q(u')\geq1/\lambda}}
		\left[\int_\R e^{i\lambda tq(u')}\Psi_l(u',t)\dint t\right]w(u')\dint u'.
	\end{align}
	Here $\Psi_l(u',t)$ denotes the same amplitude with its variables reordered.
	
	To estimate $I_{l,1}^+(\lambda)$, define the weighted sublevel volume
	\begin{equation}
		V_l(\epsilon)=\int_{\{u'\in[0,1]^{n-1}:q(u')\leq\epsilon\}}
		w(u')\dint u'.
	\end{equation}
	The substitution $u_j=v_j^{1/b_j}$ and Lemma 3.1(a) of \cite{Gre10} give
	\begin{equation}
		V_l(\epsilon)\leq C\epsilon^{1/d_l}
		(1+\ln(1/\epsilon))^{M_l-1},\qquad0<\epsilon\leq1.
	\end{equation}
	If $d_l=d_f$, then $M_l\leq k$ by \eqref{CriticalDirections} and the normal-cone argument. If $d_l<d_f$, the strictly better power of $\epsilon$ absorbs any excess logarithm. Thus, putting $p=1/d_f$, in both cases
	\begin{equation}\label{GlobalSublevel}
		V_l(\epsilon)\leq C\epsilon^p(1+\ln(1/\epsilon))^{k-1}.
	\end{equation}
	For $I_{l,1}^+(\lambda)$, first integrating the absolute value of the smooth amplitude in the bounded $t$ direction gives
	\begin{equation}
		\abs{I_{l,1}^+(\lambda)}
		\leq C V_l(1/\lambda)
		\leq C\lambda^{-p}\ln^{k-1}\lambda.
	\end{equation}
	
	For $I_{l,2}^+(\lambda)$, fix $u'$ and perform $N$ integrations by parts with respect to $t$, choosing an integer $N>1/d_l$. The derivative of the phase in $t$ is $q(u')$, independent of $t$. All derivatives therefore fall on $\Psi_l$, and its compact support on the whole real line gives no boundary terms and no negative powers of $t$. Consequently,
	\begin{equation}
		\abs{I_{l,2}^+(\lambda)}
		\leq C_N\lambda^{-N}
		\int_{\{q(u')\geq1/\lambda\}}q(u')^{-N}w(u')\dint u'.
	\end{equation}
	To estimate the last integral, we follow the proof of Lemma \ref{KeyLemma}: first integrate in a variable attaining the maximal exponent ratio, and then enlarge the domain of the remaining variables to a box. Put
	\begin{equation}
		A_l=\{j\neq s+1:a_j>0\},\qquad
		\rho_l=\min_{j\in A_l}\frac{b_j}{a_j}=\frac1{d_l},
	\end{equation}
	and choose $i\in A_l$ such that $b_i/a_i=\rho_l$. Variables with $a_j=0$ contribute only constant factors to this integral, since $b_j>0$. For the remaining variables, write
	\begin{equation}
		v=(u_j)_{j\in A_l\setminus\{i\}},\qquad
		Q(v)=\prod_{j\in A_l\setminus\{i\}}u_j^{a_j},\qquad
		L(v)=(\lambda Q(v))^{-1/a_i}.
	\end{equation}
	For fixed $v\in(0,1]^{\#A_l-1}$, the condition $q(u')\geq1/\lambda$ gives $L(v)\leq u_i\leq1$, and this interval is nonempty exactly when $Q(v)\geq1/\lambda$. On this region, since $N>\rho_l$, we have
	\begin{align}
		\int_{L(v)}^1u_i^{b_i-Na_i-1}\dint u_i
		&=\frac{(\lambda Q(v))^{N-\rho_l}-1}{a_i(N-\rho_l)}\\
		&\leq C_N(\lambda Q(v))^{N-\rho_l}.
	\end{align}
	Thus, denoting $D=\{v\in[0,1]^{\#A_l-1}:Q(v)\geq1/\lambda\}$, we obtain
	\begin{align}
		\lambda^{-N}\int_{\{q(u')\geq1/\lambda\}}q(u')^{-N}w(u')\dint u'
		&\leq C_N\lambda^{-\rho_l}
		\int_D\prod_{j\in A_l\setminus\{i\}}u_j^{b_j-\rho_l a_j-1}\dint v.
	\end{align}
	As in Lemma \ref{KeyLemma},
	\begin{equation}
		D\subset\prod_{j\in A_l\setminus\{i\}}[\lambda^{-1/a_j},1].
	\end{equation}
	Taking this larger domain yields
	\begin{align}
		\abs{I_{l,2}^+(\lambda)}
		&\leq C_N\lambda^{-\rho_l}
		\prod_{j\in A_l\setminus\{i\}}
		\int_{\lambda^{-1/a_j}}^1u_j^{b_j-\rho_l a_j-1}\dint u_j\\
		&\leq C_N\lambda^{-\rho_l}\ln^{M_l-1}\lambda.
	\end{align}
	Indeed, $b_j-\rho_l a_j\geq0$. Each factor is bounded independently of $\lambda$ when this inequality is strict, and equals $(\ln\lambda)/a_j$ when equality holds. There are exactly $M_l-1$ equality indices after the variable $u_i$ has been integrated out. If $A_l=\{i\}$, the same computation applies with the empty product and the integral in the remaining variables equal to $1$.
	If $d_l=d_f$, then $\rho_l=p$ and $M_l\leq k$. If $d_l<d_f$, then $\rho_l>p$, and the strictly better power of $\lambda$ absorbs any excess logarithm. This proves the desired estimate for $I_{l,2}^+(\lambda)$, and the other outer orthants are identical. Therefore
	\begin{equation}
		\abs{I_l(\lambda)}
		\leq C_{\Psi_l}\lambda^{-1/d_f}\ln^{k-1}\lambda.
	\end{equation}
	This includes every integer value of $1/d_f$. 
	
	Summing over all pieces completes the proof of Theorem \ref{MainTheorem}.
\end{proof}

	\begin{proof}[Proof of Theorem \ref{WeightedTheorem}]
	The weight $x^\beta$ replaces each chart exponent $\abs{\alpha^j}-1$ by
	\begin{equation}
		b_j-1,\qquad b_j=\langle\alpha^j,\beta+\mathbf1\rangle>0.
	\end{equation}
	On each fixed sign choice of the original vanishing coordinates, this is the absolute monomial weight up to a constant sign. The original nonvanishing coordinates have fixed signs, so their weight factors remain smooth and may be absorbed into the amplitude. In particular, this does not introduce a singular weight in the new transverse variable and does not require splitting its real line.
	
	Put $c_\beta=\lfloor\beta+\mathbf1\rfloor$ and $q_\beta=(\beta+\mathbf1)/c_\beta$. Let $\gamma_\beta$ be the smallest face containing $q_\beta$. Its codimension is $d_\beta$, and
	\begin{equation}
		l(\alpha^j)\leq\langle\alpha^j,q_\beta\rangle=b_j/c_\beta.
	\end{equation}
	All normals attaining equality belong to the same normal cone $N_{\gamma_\beta}$. Since its dimension is $d_\beta$, at most $d_\beta$ linearly independent chart normals attain equality. Lemma \ref{KeyLemma} therefore gives the weighted upper bound on the nonzero-unit pieces, with any strictly better local power absorbing its logarithmic factor.
	
	Near points where $f_\sigma=0$, keep the full transverse integration and use the two-region argument above. The weighted sublevel estimate becomes
	\begin{equation}
		V_l(\epsilon)\leq C\epsilon^{c_\beta}
		(1+\ln(1/\epsilon))^{d_\beta-1}.
	\end{equation}
	For the outer region on each such piece, put
	\begin{equation}
		\rho_{l,\beta}=\min_{j\in A_l}\frac{b_j}{a_j}\geq c_\beta,\qquad
		M_{l,\beta}=\#\left\{j\in A_l:\frac{b_j}{a_j}=\rho_{l,\beta}\right\},
	\end{equation}
	where $a_j=l(\alpha^j)$ and $A_l=\{j\neq s+1:a_j>0\}$ as above. Choose an integer $N>\rho_{l,\beta}$ and integrate by parts in $t$ on the whole real line. Integrating next in an outer variable attaining $\rho_{l,\beta}$ and enlarging the remaining domain to a box gives
	\begin{equation}
		\abs{I_{l,2}^+(\lambda)}
		\leq C_N\lambda^{-\rho_{l,\beta}}\ln^{M_{l,\beta}-1}\lambda.
	\end{equation}
	If $\rho_{l,\beta}=c_\beta$, the normal-cone argument gives $M_{l,\beta}\leq d_\beta$. If $\rho_{l,\beta}>c_\beta$, the gain in the power of $\lambda$ absorbs the logarithmic factor. Since there are only finitely many pieces, the integer $N$ can be chosen to satisfy all these inequalities. Combining this with the sublevel estimate for the inner region and summing over the pieces yields
	\begin{equation}
		\abs{I(\lambda)}
		\leq C_\varphi\lambda^{-c_\beta}\ln^{d_\beta-1}\lambda.
	\end{equation}
	This proves the asserted real-analytic weighted estimate.
	\end{proof}

\section*{Declaration of generative AI use}
The author wrote the first version of this article, in which the upper bound for the case $d_f\leq 1$ incurred a loss of one logarithmic factor. Following the release of GPT-6 Astra and Gemini 4-Argon, the author discussed this issue with both models. They suggested similar approaches to removing this loss, both using Stieltjes integration by parts. The author subsequently observed that the estimate could be obtained more directly by adapting the method already used in the proofs of the preceding lemmas, based on enlarging the integration domain and integrating by parts. Following the author's proposed approach, GPT-6 Astra worked out the details of the corresponding part of the proof. It also assisted with polishing the manuscript, correcting typographical errors, and supplementing the background discussion. The author takes full responsibility for the final manuscript and its mathematical content.

\section*{Acknowledgments}
	The author is supported by the National Natural Science Foundation of China (Grant No. 12501124) and the Jiangsu Natural Science Foundation (Grant No. BK20200308). The author thanks Prof. Xiaochun Li for sharing the insight that resolution of singularities plays a fundamental role in analysis. The author is grateful to Prof. Changxing Miao for inviting him to visit in 2023 and for encouraging him to pursue this meaningful topic. The author also thanks Prof. Toshihiro Nose for generously sharing his PhD thesis, from which the author learned many details about the rigorous analysis of toric resolution of singularities. This work was completed during a visit to Prof. Shaoming Guo at the Chern Institute of Mathematics; the author thanks Prof. Guo and the institute for their hospitality.


\end{document}